\documentclass[a4paper,11pt,reqno]{amsart}
\usepackage[foot]{amsaddr}
\usepackage{amsmath}
\usepackage{amsthm,enumerate, esint}
\usepackage{mathtools}
\usepackage{hyperref,url}

\usepackage{graphicx}
\usepackage{amssymb}

\usepackage{appendix}

\usepackage[italian,english]{babel}

\usepackage{fancyhdr}

\usepackage{calc}
\usepackage{url}

\usepackage[text={6in,8.6in},centering]{geometry}

\usepackage{srcltx}

\usepackage[T1]{fontenc}

\usepackage[usenames,dvipsnames]{color}

\theoremstyle{plain}
\newtheorem{theorem}{Theorem}[section]
\theoremstyle{plain}
\newtheorem{lemma}{Lemma}[section]

\newcommand{\be} {\begin{equation}}
\newcommand{\ee} {\end{equation}}
\newcommand{\bea} {\begin{eqnarray}}
\newcommand{\eea} {\end{eqnarray}}
\newcommand{\Bea} {\begin{eqnarray*}}
	\newcommand{\Eea} {\end{eqnarray*}}

\newcommand{\al} {\alpha}
\newcommand{\ba} {\beta}

\newcommand{\Ga} {\Gamma}
\newcommand{\Om} {\Omega}

\newcommand{\De} {\Delta}
\newcommand{\la} {\lambda}

\newcommand{\no} {\nonumber}

\newcommand{\lab} {\label}

\newcommand{\R}{\mathbb R}

\newcommand{\Rn}{\mathbb R^N}

\newcommand{\Hs}{\dot{H}^s(\mathbb{R}^{N})}

\newcommand{\hms}{(\dot{H}^{s})'}
\catcode`\@=11

\newcommand{\authorfootnotes}{\renewcommand\thefootnote{\@fnsymbol\c@footnote}}%

\numberwithin{equation}{section} \allowdisplaybreaks

\begin{document}
        \title[Existence of solutions]{Nonhomogeneous systems involving critical or subcritical nonlinearities}

\date{}

\author[Mousomi Bhakta]{Mousomi Bhakta\textsuperscript{1}}
\address{\textsuperscript{1}Department of Mathematics, Indian Institute of Science Education and Research, Dr. Homi Bhaba Road, Pune-411008, India}
\email{mousomi@iiserpune.ac.in}

\author[Souptik Chakraborty]{Souptik Chakraborty\textsuperscript{1}}
\email{soupchak9492@gmail.com}

\author[Patrizia Pucci]{Patrizia Pucci\textsuperscript{2}}
\address{\textsuperscript{2}Dipartimento di Matematica e Informatica, Universit\'a degli Studi di Perugia
Via Vanvitelli 1, I-06123 Perugia, Italy}
\email{patrizia.pucci@unipg.it}


\begin{abstract}
This paper deals with existence of a nontrivial {\it positive} solution to systems of equations involving
nontrivial nonhomogeneous terms and
critical
or subcritical nonlinearities.
Via a minimization argument we prove existence of a positive solution whose energy is negative provided that the  nonhomogeneous terms
 are small enough in the dual norm.
\medskip

\end{abstract}

\maketitle

\section{Introduction}

In this paper we consider the following system of equations
\begin{equation}
  \tag{$\mathcal S$}\label{sys-Q}
\left\{\begin{aligned}
		&(-\Delta)^s u +\gamma u= \frac{\al}{\al+\ba}|u|^{\al-2}u|v|^{\ba}+f(x)\;\;\text{in}\;\mathbb{R}^{N},\\
		&(-\Delta)^s v +\gamma v= \frac{\ba}{\al+\ba}|v|^{\ba-2}v|u|^{\al}+g(x)\;\;\text{in}\;\mathbb{R}^{N},\\
              & u, \, v >0\,  \mbox{ in }\,\mathbb{R}^{N},
		 \end{aligned}
  \right.
\end{equation}
where $N>2s$, $\al,\,\ba>1$, $\al+\ba\le2^*_s$, $2^*_s:=2N/(N-2s)$, $f,\, g$ are nontrivial nonnegative functionals in the dual space of $\dot{H}^s(\Rn)$ if $\al+\ba=2^*_s$ and  of ${H}^s(\Rn)$ if $\al+\ba<2^*_s$, while $\gamma=0$ if $\al+\ba=2^*_s$
and $\gamma=1$ if $\al+\ba<2^*_s$.
Here $(-\De)^s$ denotes the  fractional Laplace operator which can be defined for the Schwartz class functions
$\varphi$ as follows
\begin{equation*} \label{De-u}
  \left(-\Delta\right)^s\varphi(x): = c_{N,s}
\, \text{P.V.} \int_{\Rn}\frac{\varphi(x)-\varphi(y)}{|x-y|^{N+2s}} \, {\rm d}y, \quad c_{N,s}= \frac{4^s\Ga(N/2+ s)}{\pi^{N/2}|\Ga(-s)|}.
\end{equation*}
Let
$$\dot{H}^s(R^{N}): =\bigg\{u\in L^{2^*_s}(\R^N) \; : \; \iint_{\mathbb{R}^{2N}}\frac{|u(x)-u(y)|^2}{|x-y|^{N+2s}}\,{\rm d}x\,{\rm d}y<\infty\bigg\},$$
be the homogeneous fractional Sobolev space, endowed 	with the inner product
$\langle\cdot,\cdot\rangle_{\dot{H}^s}$ and corresponding Gagliardo norm
	$$\|u\|_{\dot{H}^{s}}:=\left( \iint_{\mathbb{R}^{2N}} \frac{|u(x)-u(y)|^2}{|x-y|^{N+2s}}\,{\rm d}x\,{\rm d}y\right)^{1/2}.$$
While $H^s(\Rn)$ is the standard fractional Sobolev Hilbert space with inner product
$\langle\cdot,\cdot\rangle_{\Hs}$ and corresponding norm
$$\|u\|_{H^{s}}:=\left(\|u\|_{2}^2+ \|u\|_{\dot{H}^{s}}^2\right)^{1/2},$$
where in general $\|\cdot\|_{p}$ is the standard norm on the Lebesgue  space $L^p(\Rn)$, when $p\ge1$.

In the vectorial case, the natural solution space for~\eqref{sys-Q} is  the Hilbert space $\Hs\times\Hs$, equipped with the inner product
$$\big\langle (u,v), (\phi,\psi)\big\rangle_{\dot{H}^s\times\dot{H}^s}:=\langle u,\phi\rangle_{\dot{H}^s}+\langle v,\psi\rangle_{\dot{H}^s},$$ and the norm
$$\|(u,v)\|_{\dot{H}^s\times\dot{H}^s}:=\big(\|u\|^2_{\dot{H}^s}+\|v\|^2_{\dot{H}^s}\big)^\frac{1}{2},$$
when $\al+\ba=2^*_s$, while is $H^s(\Rn)\times H^s(\Rn)$ equipped with the inner product
$$\big\langle (u,v), (\phi,\psi)\big\rangle_{H^s\times H^s}:=\langle u,\phi\rangle_{\dot{H}^s}+\langle v,\psi\rangle_{\dot{H}^s}+ \langle u,\phi\rangle_{L^2}+\langle v,\psi\rangle_{L^2},$$ and the norm
$$\|(u,v)\|_{H^s\times H^s}:=\big(\|u\|^2_{H^s}+\|v\|^2_{H^s} \big)^\frac{1}{2},$$
if $\al+\ba<2^*_s$.

When  $\al+\ba=2^*_s$, we say $(u,v)\in\dot{H}^s(\Rn)\times\dot{H}^s(\Rn)$ is a {\em solution}
of \eqref{sys-Q} if $u,\, v>0$ in $\Rn$ and
\begin{align*}
\big\langle (u,v), (\phi,\psi)\big\rangle_{\dot{H}^s\times\dot{H}^s}&= \frac{\al}{2^*_s}\int_{\R{^N}}|u|^{\al-2}u|v|^{\ba}\phi\,{\rm d}x+\frac{\ba}{2^*_s}\int_{\R{^N}}|v|^{\ba-2}v|u|^{\al}\psi\,{\rm d}x\\
&\qquad\qquad\qquad+ \prescript{}{(\dot{H}^s)'}{\langle}f,\phi{\rangle}_{\dot{H}^s} +\prescript{}{(\dot{H}^s)'}{\langle}g,\psi{\rangle}_{\dot{H}^s}
\end{align*}
holds for every $(\phi,\psi)\in\dot{H}^s(\Rn)\times\dot{H}^s(\Rn)$, while if $\al+\ba<2^*_s$ a couple
$(u,v)\in{H}^s(\Rn)\times{H}^s(\Rn)$ is said to be a
{\em solution}
of \eqref{sys-Q} if $u,\, v>0$ in $\Rn$ and
\begin{align*}
\big\langle (u,v), (\phi,\psi)\big\rangle_{{H}^s\times{H}^s}&= \frac{\al}{\al+\ba}\int_{\R{^N}}|u|^{\al-2}u|v|^{\ba}\phi\,{\rm d}x+\frac{\ba}{\al+\ba}\int_{\R{^N}}|v|^{\ba-2}v|u|^{\al}\psi\,{\rm d}x\\
&\qquad\qquad\qquad+ \prescript{}{H^{-s}}{\langle}f,\phi{\rangle}_{H^s}
+\prescript{}{H^{-s}}{\langle}g,\psi{\rangle}_{{H}^s}
\end{align*}
holds for every $(\phi,\psi)\in {H}^s(\Rn)\times {H}^s(\Rn)$.

 When the domain is a open bounded subset of $\Rn$, in a pioneering work, Tarantello \cite{T} proves existence of two positive solutions for the nonhomogeneous problem
\be\lab{1-12-1}-\De u=|u|^{2^*-2}u+f \mbox{ in }\, \Omega, \quad u=0 \mbox{ on }\, \partial\Omega,
\quad 2^*=\frac{2N}{N-2},\ee where $0\leq f\in H^{-1}(\Om)$ satisfies suitable condition. In \cite{CDPM, Mu} the authors study existence of sign changing solutions of \eqref{1-12-1}. In \cite{BP}, the first and third author of the current paper treat the scalar version of~$(\mathcal{S})$ with the critical nonlinearity, namely the equation:
\begin{equation*}
\left\{\begin{aligned}
		&(-\Delta)^s u = a(x) |u|^{2^*_s-2}u+f(x)\;\;\text{in}\;\mathbb{R}^{N},\\
		&u >0 \quad\text{in}\quad\mathbb{R}^{N},\quad
		u \in \dot{H}^{s}{(\mathbb{R}^{N})},
		 \end{aligned}
  \right.
\end{equation*}
where $0< a\in L^\infty(\Rn)$,\, $a(x)\to 1$ as $|x|\to\infty$ and $f\in \dot{H}^s(\Rn)'$ and prove  existence of at least two positive solutions when $\|f\|_{(\dot{H}^s)'}$ is small enough. For the scalar version of~$(\mathcal{S})$, with subcritical nonlinearities, we refer to~\cite{AT, CZ, J, Z}  in the local case and  to~\cite{BCG}
in the nonlocal case. In all these papers   existence of at least two positive solutions is actually proved.
\medskip

Elliptic systems arise in biological applications (e.g. population dynamics) or physical applications (e.g. models of a nuclear reactor) and have been drawn a lot of attention (see \cite{AMS, CFMT, LW, M, QS, RZ} and references therein). For systems in bounded domains with nonhomogeneous terms we refer to~\cite{BR}.  In the case of vector valued solutions for Schr\"{o}dinger systems of equations in $\R^3$ with nonhomogeneous perturbation, we refer to~\cite{LP}, where the authors have applied Lyapunov--Schmidt reduction scheme to construct multiple solutions.

In the nonlocal case, there are not so many papers, in which weakly coupled systems of equations have been studied. To quote a few, we refer to~\cite{BN, CS, CMSY, FMPSZ, GMS, HSZ}. Actually all these papers deal with Dirichlet systems of equations in bounded domains. For the nonlocal systems of equations in the entire space $\Rn$, we cite \cite{FPS, FPZ}
and the references therein.   To the best of our knowledge, so far there have been no papers in the literature, where existence of nontrivial solutions to system of equations, with fractional Laplacian and the critical or subcritical exponents in $\Rn$, have  been established in the
nontrivial nonhomogeneous case. The main result in the paper is new even in the local case $s=1$ and is stated below,
where ker$(f)$ denotes the kernel of $f$.

\begin{theorem}\lab{th:ex-f} $(i)$ If $\al+\ba=2^*_s$, and
$f,\, g$ are nontrivial nonnegative functionals in the dual space $\dot{H}^s(\Rn)'$ of $\dot{H}^s(\Rn)$ such that {\rm ker}$(f)$= {\rm ker}$(g)$, then system~\eqref{sys-Q} admits a nontrivial solution $(\bar u, \bar v)$ such that $\bar u>0$ and $\bar v>0$,
provided that $0<\max\{\|f\|_{\hms},\|g\|_{\hms}\}\leq d$ for some $d>0$ sufficiently small.\smallskip

$(ii)$ If $\al+\ba<2^*_s$, and
$f,\, g$ are nontrivial nonnegative functionals in the dual space $H^{-s}(\Rn)$ of $H^s(\Rn)$ such that
{\rm ker}$(f)$={\rm ker}$(g)$, then~\eqref{sys-Q} admits a nontrivial solution $(\bar u, \bar v)$ such that $\bar u>0$ and $\bar v>0$,
provided that $0<\max\{\|f\|_{H^{-s}},\|g\|_{H^{-s}}\}\leq d$ for some $d>0$ sufficiently small.\smallskip

Furthermore, in both the cases $(i)$ and $(ii)$ if $f\equiv g$, then the solution $(\bar u, \bar v)$ has the property that
$\bar u\not\equiv \bar v$, whenever $\al\neq\ba$. Finally, if $\al=\ba$ but $f\not\equiv g$, then $\bar u\not\equiv \bar v$.
\end{theorem}

Let us emphasize that here we introduce
suitable assumptions under which system \eqref{sys-Q} admits a positive solution, with different components,
while papers devoted to systems seem not to address this question at all. Therefore, we actually solve 
system \eqref{sys-Q} when it does not reduce into a single equation.

To the best of our knowledge, the question of finding at least two nontrivial solutions to \eqref{sys-Q} remains open in the vectorial case. In the scalar case we are able to exhibit existence of two different solutions  in the recent papers~\cite{BCG, BP}.

%

\section{Proof of Theorem \ref{th:ex-f}}

Before proving the main Theorem~\ref{th:ex-f} let us present some useful notation and auxiliary results.
Define
\be\lab{S}
S=\inf_{u\in\dot{H}^s(\Rn)\setminus\{0\}}\frac{\|u\|^2_{\dot{H}^s}}
{\|u\|_{2^*_s}^2},\qquad S_{\al+\ba}=\inf_{u\in{H}^s(\Rn)\setminus\{0\}}\frac{\|u\|^2_{\dot{H}^s}}
{\|u\|_{\al+\ba}^2},\no\ee
and
$$S_{(\al,\ba)}=\begin{cases}\displaystyle{\inf_{(u,v)\in \dot{H}^s(\Rn)\times\dot{H}^s(\Rn) \setminus\{(0,0)\}} \frac{\|u\|^2_{\dot{H}^s}+ \|v\|^2_{\dot{H}^s}}{\bigg(\displaystyle\int_{\Rn}|u|^{\al}|v|^{\ba}{\rm d}x\bigg)^{2/2^*_s}}},&\mbox{if }\al+\ba=2^*_s\\
\phantom{a}\\
\displaystyle{\inf_{(u,v)\in {H}^s(\Rn)\times{H}^s(\Rn) \setminus\{(0,0)\}} \frac{\|u\|^2_{\dot{H}^s}+ \|v\|^2_{\dot{H}^s}}{\bigg(\displaystyle\int_{\Rn}|u|^{\al}|v|^{\ba}{\rm d}x\bigg)^{2/(\al+\ba)}}},&\mbox{if
}\al+\ba<2^*_s.
\end{cases}$$
In the celebrated paper \cite{CLO}  Chen, Li and Ou   prove that
when $\al+\ba=2^*_s$ the Sobolev constant $S_{\al+\ba}=S$ is achieved by $w$, where $w$ is the unique positive solution (up to  translations and dilations) of
$$(-\Delta)^sw = w^{2^*_s-1}\;\;\text{in}\;\mathbb{R}^{N},\quad
w \in \dot{H}^{s}{(\mathbb{R}^{N})}.$$
Indeed, any positive solution of the above equation  is radially symmetric, with respect to some point $x_0\in\Rn$, strictly decreasing in $r=|x-x_0|$, of class $C^{\infty}(\Rn)$ and so of the explicit parametric form
$$w(x)= c_{N,s}\bigg(\frac{\la}{\la^2+|x-x_0|^2}\bigg)^\frac{N-2s}{2},$$
for some $\la>0$. On the other hand, when $2<\al+\ba<2^*_s$, Frank, Lenzmann and Silvestre in their celebrated paper \cite{FLS} prove that $S_{\al+\ba}$ is achieved by unique (up to a translation) positive ground state solution $w$ of
$$(-\Delta)^sw+w = w^{\al+\ba-1}\;\;\text{in}\;\mathbb{R}^{N},\quad
w \in H^{s}{(\mathbb{R}^{N})}.$$
Furthermore,  $w$  is radially symmetric, symmetric decreasing $C^{\infty}(\Rn)$ function which satisfies the following decay property in $\mathbb R^N$
		$$\frac{C^{-1}}{1+|x|^{N+2s}}\leq w(x)\leq \frac{C}{1+|x|^{N+2s}},$$
with some constant $C>0$ depending on $N,\;\al+\ba,\;s$.

\begin{lemma}\lab{l:s1}
There exists a positive constant $C=C(\al,\ba,s,N)$ such that when $\al+\ba=2^*_s$
$$\bigg(\int_{\Rn}|u|^{\al}|v|^{\ba} {\rm d}x\bigg)^{1/2^*_s}\leq C
\|(u,v)\|_{\dot{H}^s\times\dot{H}^s}$$
for all $(u,v)\in\dot{H}^s(\Rn)\times\dot{H}^s(\Rn)$, while if $\al+\ba<2^*_s$
$$\bigg(\int_{\Rn}|u|^{\al}|v|^{\ba} {\rm d}x\bigg)^{1/(\al+\ba)}\le C
\|(u,v)\|_{{H}^s(\Rn)\times{H}^s(\Rn)}$$
for all $(u,v)\in{H}^s(\Rn)\times{H}^s(\Rn)$.
\end{lemma}

\begin{proof}
It easily follows from the definition of $S_{\al+\ba}$ and the inequality
$$|t|^{\al}|\tau|^{\ba}\leq |t|^{\al+\ba}+|\tau|^{\al+\ba}$$
for all $(t,\tau)\in\mathbb R^2$.
\end{proof}

Next, we recall a result from \cite{FMPSZ} which states the relation between $S_{(\al,\ba)}$ and $S_{\al+\ba}$.

\begin{lemma}\lab{l:S}\cite[Lemma 5.1]{FMPSZ} In all cases $\al>1$, $\ba>1$, with $\al+\ba\le2^*_s$, it results
$$S_{(\al,\ba)}=\bigg[\left(\frac{\al}{\ba}\right)^\frac{\ba}{\al+\ba}
+\left(\frac{\al}{\ba}\right)^\frac{-\al}{\al+\ba} \bigg]S_{\al+\ba}.$$
Moreover, if $w$ achieves $S_{\al+\ba}$ then $(Bw, Cw)$ achieves $S_{(\al,\ba)}$ for all positive constants $B$ and $C$ such that $B/C=\sqrt{\al/\ba}$.
\end{lemma}

Finally we prove a short useful result

\begin{lemma}\label{Sab} In all cases $\al>1$, $\ba>1$, with $\al+\ba\le2^*_s$,
\begin{equation*}S_{(\al,\ba)}>S_{\al+\ba}\end{equation*}
holds true.
\end{lemma}

\begin{proof}
If $\al>\ba$, then using Lemma~\ref{l:S},
	$$\frac{S_{(\al,\ba)}}{S_{\al+\ba}}=\left(\frac{\al}{\ba}\right)^\frac{\ba}{\al+\ba}
+\left(\frac{\al}{\ba}\right)^\frac{-\al}{\al+\ba}=\left(\frac{\al}{\ba}\right)^\frac{\ba}{\al+\ba}\frac{\al+\ba}{\al}>1.$$
  Similarly, if $\al<\ba$ then
\begin{align*}
\frac{S_{(\al,\ba)}}{S}=\left(\frac{\al}{\ba}\right)^\frac{\ba}{\al+\ba}
+\left(\frac{\al}{\ba}\right)^\frac{-\al}{\al+\ba}&=\left(\frac{\ba}{\al}\right)^{-\frac{\ba}{\al+\ba}}
+\left(\frac{\ba}{\al}\right)^\frac{\al}{\al+\ba}=\left(\frac{\ba}{\al}\right)^\frac{\al}{\al+\ba}
\bigg[1+\left(\frac{\ba}{\al}\right)^{-1}\bigg]\\
&=\left(\frac{\ba}{\al}\right)^\frac{\al}{\al+\ba}
\frac{\al+\ba}{\ba}>1.\end{align*}
Further,  $S_{(\al,\ba)}>2S_{\al+\ba}$ for $\al=\ba$.
\end{proof}

We are finally in a position to prove the main result and we simply say that a couple $(u,v)$ is positive if both components are positive.

\begin{proof}[Proof of Theorem~$\ref{th:ex-f}$ -- Part $(i)$.]
Let $\al+\ba=2^*_s$. We note that system \eqref{sys-Q} is variational and the underlying functional  is
\begin{equation*}\lab{11-20-1}
I_{f,g}(u,v):= \frac{1}{2}\|(u,v)\|^2_{\dot{H}^s\times\dot{H}^s}-\frac{1}{2^*_s}\int_{\R^N}|u|^{\al}|v|^{\ba}\,{\rm d}x -\prescript{}{(\dot{H}^s)'}{\langle}f,u{\rangle}_{\dot{H}^s}-\prescript{}{(\dot{H}^s)'}{\langle}g,v{\rangle}_{\dot{H}^s},
\end{equation*}
which is well defined in $\dot{H}^s(\Rn)\times\dot{H}^s(\Rn)$ and of class
$C^1\big(\dot{H}^s(\Rn)\times\dot{H}^s(\Rn)\big)$. Moreover, if $(u,v)$ is a solution of~\eqref{sys-Q}, then $(u,v)$ is a positive critical point of $I_{f,g}$ and  vice versa.

Let us now introduce the auxiliary functional
\begin{equation*}\lab{11-20-1'}
J_{f,g}(u,v):= \frac{1}{2}\|(u,v)\|^2_{\dot{H}^s\times\dot{H}^s}
-\frac{1}{2^*_s}\int_{\R^N}u_+^{\al}v_+^{\ba}\,{\rm d}x -\prescript{}{(\dot{H}^s)'}{\langle}f,u{\rangle}_{\dot{H}^s}-\prescript{}{(\dot{H}^s)'}{\langle}g,v{\rangle}_{\dot{H}^s},
\end{equation*}
which is well defined in $\dot{H}^s(\Rn)\times\dot{H}^s(\Rn)$ and of class $C^1\big(\dot{H}^s(\Rn)\times\dot{H}^s(\Rn)\big)$,
with second derivative.
Indeed, for all $(u,v),\,(\phi,\psi)\in\dot{H}^s(\Rn)\times\dot{H}^s(\Rn)$
\begin{equation}\label{esti-1}
\begin{aligned}
	J''_{f,g}(u,v)\big((\phi,\psi), (\phi,\psi)\big) &=  \|(\phi,\psi)\|_{\dot{H}^s\times\dot{H}^s}^2 -\frac{\al(\al-1)}{2^*_s}\int_{\Rn}u_+^{\al-2}v_+^{\ba}\phi^2 \, {\rm d}x\\
	&\quad -\frac{\ba(\ba-1)}{2^*_s}\int_{\Rn}u_+^{\al}v_+^{\ba-2}\psi^2 \, {\rm d}x-\frac{2\al\ba}{2^*_s}\int_{\Rn}u_+^{\al-1}v_+^{\ba-1}\phi\psi \, {\rm d}x.\\
		\end{aligned}
\end{equation}
	Using H\"{o}lder's and Sobolev's inequalities, we estimate the second term on the RHS as follows
	\begin{align*}
	\int_{\Rn}u_+^{\al-2}v_+^{\ba}\phi^2  {\rm d}x &\leq \left( \int_{\Rn} |\phi|^{2^*_s}  {\rm d}x\right)^{\frac{2}{2^*_s}}
	\left(\int_{\Rn} |u|^{2^*_s}  {\rm d}x \right)^{\frac{\al-2}{2^*_s}}\left( \int_{\Rn} |v|^{2^*_s}  {\rm d}x\right)^{\frac{\ba}{2^*_s}}\\
	&\leq S^{-1-\frac{\al-2}{2}-\frac{\ba}{2}} \, \|u\|^{\al-2}_{\dot{H}^s}\|v\|^{\ba}_{\dot{H}^s} \|\phi\|^{2}_{\dot{H}^s}\\
	&\leq S^{-\frac{2^*_s}{2}}\|(u,v)\|^{2^*_s-2}_{\dot{H}^s\times\dot{H}^s}\|(\phi,\psi)\|^{2}_{\dot{H}^s\times\dot{H}^s}.
	\end{align*}	
 In the inequality we have used the fact that $\|u\|_{\dot{H}^s}\leq \|(u,v)\|_{\dot{H}^s\times\dot{H}^s}$ and $\al+\ba=2^*_s$.
Similarly, $$\int_{\Rn}u_+^{\al}v_+^{\ba-2}\psi^2 \, {\rm d}x	\leq S^{-\frac{2^*_s}{2}}\|(u,v)\|^{2^*_s-2}_{\dot{H}^s\times\dot{H}^s}\|(\phi,\psi)\|^{2}_{\dot{H}^s\times\dot{H}^s}.$$
Furthermore,
\begin{align*}
\int_{\Rn}u_+^{\al-1}v_+^{\ba-1}\phi\psi \, {\rm d}x &\leq\left( \int_{\Rn} |\phi|^{2^*_s} \, {\rm d}x\right)^{\frac{1}{2^*_s}}\left( \int_{\Rn} |\psi|^{2^*_s} \, {\rm d}x\right)^{\frac{1}{2^*_s}}
	\left(\int_{\Rn} |u|^{2^*_s} \, {\rm d}x \right)^{\frac{\al-1}{2^*_s}}\left( \int_{\Rn} |v|^{2^*_s} \, {\rm d}x\right)^{\frac{\ba-1}{2^*_s}}	\\
&\leq S^{-\frac{1}{2}-\frac{1}{2}-\frac{\al-1}{2}-\frac{\ba-1}{2}}\|\phi\|_{\dot{H}^s}\|\psi\|_{\dot{H}^s}\|u\|^{\al-1}_{\dot{H}^s}\|v\|^{\ba-1}_{\dot{H}^s} 	\\
&\leq \frac{S^{-\frac{2^*_s}{2}}}{2}\|(\phi,\psi)\|_{\dot{H}^s\times\dot{H}^s}^2\|(u,v)\|_{\dot{H}^s\times\dot{H}^s}^{2^*_s-2}.
\end{align*}		
Thus, substituting the above three estimates in \eqref{esti-1}, we obtain
\begin{align*}
J''_{f,g}(u,v)\big((\phi,\psi), (\phi,\psi)\big) &\geq \left( 1 - \frac{S^{-\frac{2^*_s}{2}}}{2^*_s} \|(u,v)\|^{2^*_s-2}_{\dot{H}^s\times\dot{H}^s}\big[\al(\al-1)+\ba(\ba-1)+\al\ba\big]\right)\cdot\\
	&\qquad\times\|(\phi,\psi)\|_{\dot{H}^s\times\dot{H}^s}^2 .
\end{align*}
Therefore, 		
	$J''_{f,g}(u,v)$ is positive definite for $(u,v)$ in the ball centered at 0 and of radius $r$ in $\dot{H}^s(\Rn)\times\dot{H}^s(\Rn)$, where
	$$r =\left(\frac{2^*_s}{\al^2+\ba^2+\al\ba-2^*_s}\right)^{\frac{1}{2^*_s-2}}S^\frac{N}{4s}.$$
Hence $J_{f,g}$ is strictly convex in~$B_r$. For $(u,v)\in\Hs\times\Hs$, with $\|(u,v)\|_{\dot{H}^s\times\dot{H}^s} = r$,
\begin{align*}
	J_{f,g}(u,v) &= \frac{1}{2}\|(u,v\|_{\dot{H}^s\times\dot{H}^s}^2 - \frac{1}{2^*_s}\int_{\Rn}u_+^{\al}v_+^{\ba} {\rm d}x-  \prescript{}{\hms}{\langle}f, u{\rangle}_{\dot{H}^s}-  \prescript{}{\hms}{\langle}g, v{\rangle}_{\dot{H}^s}\\
	&\geq\bigg(\frac{1}{2}-\frac{1}{2^*_s}S_{(\al,\ba)}^{-\frac{2^*_s}{2}}  r^{2^*_s-2}\bigg) r^2 - (\|f\|_{\hms}\|u\|_{\dot{H}^s}+\|g\|_{\hms}\|v\|_{\dot{H}^s})\\
	&\geq\bigg(\frac{1}{2}-\frac{1}{2^*_s}S_{(\al,\ba)}^{-\frac{2^*_s}{2}}  r^{2^*_s-2}\bigg) r^2 - \big(\|f\|_{\hms}+\|g\|_{\hms}
\big)r.
\end{align*}
As $r^{2^*_s-2}  =  \frac{2^*_s}{\al^2+\ba^2+\al\ba-2^*_s} S^{\frac{2^*_s}{2}}$, we obtain	
	\be\lab{2-18-7}
	J_{f,g}(u,v)\geq \bigg[\frac{1}{2} -\frac{1}{\al^2+\ba^2+\al\ba-2^*_s}\bigg(\frac{S}{S_{(\al,\ba)}}\bigg)^\frac{2^*_s}{2} \bigg]r^2-r(\|f\|_{\hms}+\|g\|_{\hms}).
	\ee
We claim that
\be\lab{2-18-4}(\al^2+\ba^2+\al\ba-2^*_s)\bigg(\frac{S_{(\al,\ba)}}{S}\bigg)^{2^*_s/2}>2. \ee
 By Lemma~\ref{Sab} and $\al+\ba=2^*_s$, we have
$$
(\al^2+\ba^2+\al\ba-2^*_s)\bigg(\frac{S_{(\al,\ba)}}{S}\bigg)^\frac{2^*_s}{2}> (\al^2+\ba^2+\al\ba-2^*_s)\frac{S_{(\al,\ba)}}{S}=\big[2^*_s(2^*_s-1)-\al\ba\big]\left(\frac{\al}{\ba}
\right)^\frac{\ba}{2^*_s}\frac{2^*_s}{\al}.
$$
Since $2^*_s>2$, to prove \eqref{2-18-4} it is enough to show that
$$
\big[2^*_s(2^*_s-1)-\al\ba\big]\left(\frac{\al}{\ba}\right)^\frac{\ba}{2^*_s}\frac{1}{\al}>1.
$$
Now,
\begin{align*}
\big[2^*_s(2^*_s-1)-\al\ba\big]\left(\frac{\al}{\ba}\right)^\frac{\ba}{2^*_s}\frac{1}{\al}>1
 &\Longleftrightarrow 2^*_s(2^*_s-1)-\al\ba>\ba^\frac{\ba}{2^*_s}\al^\frac{2^*_s-\ba}{2^*_s}\\
 &\Longleftrightarrow 2^*_s(2^*_s-1)>\al\ba\bigg[1+ \frac{1}{\al^\frac{\ba}{2^*_s}\ba^\frac{2^*_s-\ba}{2^*_s}}\bigg].
\end{align*}
Since, $\al,\,\ba>1$ and $\al+\ba=2^*_s$, we have
$$
\al\ba\bigg[1+ \frac{1}{\al^\frac{\ba}{2^*_s}\ba^\frac{2^*_s-\ba}{2^*_s}}\bigg]<2\al\ba  \leq \frac{(\al+\ba)^2}{2}=\frac{(2^*_s)^2}{2}<2^*_s(2^*_s-1).
$$
Hence the claim \eqref{2-18-4} follows.

Now, by  \eqref{2-18-7} and \eqref{2-18-4} there exists
a number $d>0$ such that
	$$
	\inf_{\|(u,v)\|_{\dot{H}^s\times\dot{H}^s}  =  r}J_{f,g}(u,v) >0, \quad  \mbox{provided that } \  0<\max\{\|f\|_{\hms},\|g\|_{\hms}\}\leq d.
	$$
	Furthermore,  for $(u,v)\in \dot{H}^s(\Rn)\times \dot{H}^s(\Rn)$, with $u>0$ and $v\geq 0$,
\begin{equation}\lab{2-24-1}
	J_{f,g}(tu, tv)
\left\{\begin{aligned}
	&<0 \, \mbox{ for }\,  t>0 \, \mbox{ small enough}\\
	&>0  \, \mbox{ for }\,  t<0 \, \mbox{ small enough},
	 \end{aligned}
  \right.
\end{equation}
since $f$ and $g$ are nontrivial. Combining this along with the fact that $J_{f,g}$ is strictly convex in $B_r$ and
$$\inf_{\|(u,v)\|_{\dot{H}^s\times\dot{H}^s} \, = \, r}J_{f,g}(u,v)  >0=J_{f,g}(0,0),$$ we conclude that there exists a unique critical point $(\bar u, \bar v)$ of $J_{f,g}$ in $B_r$ such that
$$
	J_{f,g}(\bar u, \bar v) = \inf_{\|(u,v)\|_{\dot{H}^s\times\dot{H}^s}<r}J_{f,g}(u,v)<J_{f,g}(0,0)=0.
$$
Therefore, $(\bar u, \bar v)$ is a nontrivial solution of
\begin{equation}
\label{sys-Q"}
\left\{\begin{aligned}
		&(-\Delta)^s u = \frac{\al}{2^*_s}u_+^{\al-1}v_+^{\ba}+f(x)\;\;\text{in}\;\mathbb{R}^{N},\\
		&(-\Delta)^s v = \frac{\ba}{2^*_s}v_+^{\ba-1}u_+^{\al}+g(x)\;\;\text{in}\;\mathbb{R}^{N},\\
  &u,\, v \in \dot{H}^{s}{(\mathbb{R}^{N})}.
		 \end{aligned}
  \right.
\end{equation}
Since,  $f$ and $g$ are nonnegative functionals, then taking $(\phi,\psi)=(\bar u_-, \bar v_-)$ as a test function in \eqref{sys-Q"}, we obtain
\Bea
&&-\|\bar u_-\|^2_{\dot{H}^s}  - \iint_{\R^{2N}} \frac{[\bar u_+(y)\bar u_-(x)+\bar u_+(x)\bar u_-(y)]}{|x - y|^{N + 2s}}  {\rm d}x \, {\rm d}y -\|\bar v_-\|^2_{\dot{H}^s}  \\
&&\qquad\quad -\iint_{\R^{2N}} \frac{[\bar v_+(y)\bar v_-(x)+\bar v_+(x)\bar v_-(y)]}{|x - y|^{N + 2s}}  {\rm d}x \, {\rm d}y
=\prescript{}{H^{-s}}{\langle}f,\bar u_-{\rangle}_{H^s}+ \prescript{}{H^{-s}}{\langle}g,\bar v_-{\rangle}_{H^s}\geq 0.
\Eea
 This in turn implies $\bar u_-=0$ and $\bar v_-=0$, i.e., $\bar u\geq 0$ and $\bar v\geq 0$. Therefore, $(\bar u, \bar v)$ is nontrivial nonnegative solution of~\eqref{sys-Q}.

 Next we assert that $(\bar u, \bar v)\neq (0,0)$ implies $\bar u\neq 0$ and $\bar v\neq 0$.
 Suppose not, that is assume for instance that $\bar u\neq 0$ but $\bar v=0$. Then taking the test function $(\phi,\psi)=(\bar u, 0)$ we get $$\|\bar u\|^2_{\dot{H}^s}=\prescript{}{(\dot{H}^s)'}{\langle}f,\bar u{\rangle}_{\dot{H}^s}.$$
 Next, choose as test function $(\phi,\psi)=(0, \bar u)$, so that
 $$\prescript{}{(\dot{H}^s)'}{\langle}g, \bar u{\rangle}_{\dot{H}^s}=0.$$
Hence, $\|\bar u\|_{\dot{H}^s}=0$, since
ker$(f)$ = ker$(g)$ by assumption. This contradicts the fact that $(\bar u, \bar v)\neq (0,0)$. Similarly, we can show that if $\bar u=0$ then $\bar v=0$ too. Hence the assertion follows.

Let us claim that $\bar u>0$ and $\bar v>0$ in $\Rn$.
To prove the claim, first we note that taking the test function $(\phi,\psi)=(\phi,0)$, where $\phi\in \Hs$ with $\phi\geq 0$, we obtain
$$\langle \bar u,\phi\rangle_{\dot{H}^s} =\frac{\al}{2^*_s}\int_{\R{^N}}\bar u^{\al-1}\bar v^{\ba}\phi\,{\rm d}x+ \prescript{}{(\dot{H}^s)'}{\langle}f,\phi{\rangle}_{\dot{H}^s}\geq 0,$$
as $f$ is a nonnegative functional and $\bar u,\, \bar v\geq 0$. This implies $\bar u$ is a weak supersolution to $$(-\De)^s u=0.$$
Therefore, applying  the maximum principle \cite[Theorem~1.2 $(ii)$]{DPQ}, with $c\equiv 0$ and $p=2$ there, it follows that $\bar u>0$ in $\Rn$. Similarly, taking the test function $(\phi,\psi)=(0, \psi)$, with $\psi\in \Hs$ and $\psi\geq 0$, yields $\bar v>0$ in $\Rn$. This proves the claim.

The final assertion will be shown below by the method of contradiction. Therefore, let us suppose $\bar u\equiv \bar v$ and divide the proof in the two cases covered by the theorem.

First, we assume $f\equiv g$ but $\al\neq \ba$. Then, taking the test function $(\phi,\psi)=(\bar u,-\bar u)$ yields
$$\frac{1}{2^*_s}(\al-\ba)\int_{\Rn}\bar u^{\al+\ba}dx=0.$$
This is impossible since $\bar u$ is positive in $\Rn$.

In the remaining case, we assume $\al=\ba$ but $f\not\equiv g$
and ker$(f)=$ker$(g)$. Then taking the test function $(\phi,\psi)=(\phi,-\phi)$, where $\phi\in C^\infty_0(\Rn)$, we obtain
$$ \prescript{}{(\dot{H}^s)'}{\langle}f-g,\phi{\rangle}_{\dot{H}^s}=0.$$
This in turn implies $f\equiv g$ as $\phi\in C^\infty_0(\Rn)$ is arbitrary. This contradiction completes the proof
of Part~$(i)$.\medskip

{\em Part $(ii)$.}
The proof follows along the same lines as in Part $(i)$, therefore we just mention only the differences. It is easy to see that the associated functional corresponding to \eqref{sys-Q} is now
$$\tilde I_{f,g}(u,v):= \frac{1}{2}\|(u,v)\|^2_{H^s\times H^s}-\frac{1}{\al+\ba}\int_{\R^N}|u|^{\al}|v|^{\ba}\,{\rm d}x -\prescript{}{H^{-s}}{\langle}f,u{\rangle}_{H^s}-\prescript{}{H^{-s}}{\langle}g,v{\rangle}_{H^s}.$$
Let us introduce the auxiliary functional as
\begin{equation*}\lab{11-20-1"}
\tilde J_{f,g}(u,v):= \frac{1}{2}\|(u,v)\|^2_{H^s\times H^s}
-\frac{1}{\al+\ba}\int_{\R^N}u_+^{\al}v_+^{\ba}\,{\rm d}x -\prescript{}{H^{-s}}{\langle}f,u{\rangle}_{H^s}-\prescript{}{H^{-s}}{\langle}g,v{\rangle}_{H^s},
\end{equation*}
which is well defined in $H^s(\Rn)\times H^s(\Rn)$ and of class $C^1\big(H^s(\Rn)\times H^s(\Rn))$, with second derivative. Arguing as before, we obtain
for all $(u,v),\,(\phi,\psi)\in H^s(\Rn)\times H^s(\Rn) $
\begin{align*}
	\tilde J''_{f,g}(u,v)\big((\phi,\psi), (\phi,\psi)\big) &=  \|(\phi,\psi)\|_{H^s\times H^s}^2 -\frac{\al(\al-1)}{\al+\ba}\int_{\Rn}u_+^{\al-2}v_+^{\ba}\phi^2  {\rm d}x\\
	&\qquad -\frac{\ba(\ba-1)}{\al+\ba}\int_{\Rn}u_+^{\al}v_+^{\ba-2}\psi^2  {\rm d}x-\frac{2\al\ba}{\al+\ba}\int_{\Rn}u_+^{\al-1}v_+^{\ba-1}\phi\psi  {\rm d}x.\\
	 & \geq  \left( 1 - \frac{S_{\al+\ba}^{-\frac{\al+\ba}{2}}}{\al+\ba} \|(u,v)\|^{\al+\ba-2}_{H^s\times H^s }\big[\al(\al-1)+\ba(\ba-1)+\al\ba\big]\right)\cdot\\
	&\qquad\qquad\qquad\times\|(\phi,\psi)\|_{H^s\times H^s}^2 .
\end{align*}
Therefore,
$\tilde J''_{f,g}(u,v)$ is positive definite for $(u,v)$ in the ball centered at 0 and of radius $r$ in $H^s(\Rn)\times H^s(\Rn)$, where
	$$r =\left(\frac{\al+\ba}{\al^2+\ba^2+\al\ba-(\al+\ba)}\right)^{\frac{1}{\al+\ba-2}}S_{\al+\ba}^\frac{\al+\ba}{2(\al+\ba-2)}.$$
Hence $\tilde J_{f,g}$ is strictly convex in~$B_r$. Furthermore, for all $(u,v)\in H^s(\Rn)\times H^s(\Rn) $, with $\|(u,v)\|_{H^s\times H^s} = r$,
\be\lab{2-27-1}\tilde J_{f,g}(u,v)\geq \bigg[\frac{1}{2} -\frac{1}{\al^2+\ba^2+\al\ba-(\al+\ba)}\bigg(\frac{S_{\al+\ba}}{S_{(\al,\ba)}}\bigg)^\frac{\al+\ba}{2} \bigg]r^2-r(\|f\|_{H^{-s}}+\|g\|_{H^{-s}}).\ee
Since $S_{(\al,\ba)}>S_{\al+\ba}$
by Lemma~\ref{Sab}, we have
\begin{align*}
\big(\al^2+\ba^2+\al\ba-(\al+\ba)\big)\bigg(\frac{S_{(\al,\ba)}}{S_{\al+\ba}}\bigg)^{(\al+\ba)/2} &\geq \big(\al^2+\ba^2+\al\ba-(\al+\ba)\big)\frac{S_{(\al,\ba)}}{S_{\al+\ba}}\\
&=\big[(\al+\ba)(\al+\ba-1)-\al\ba\big]\left(\frac{\al}{\ba}
\right)^\frac{\ba}{\al+\ba}\frac{\al+\ba}{\al}.
\end{align*}
Therefore, to prove $$\big[\al^2+\ba^2+\al\ba-(\al+\ba)\big]
\bigg(\frac{S_{(\al,\ba)}}{S_{\al+\ba}}\bigg)^{(\al+\ba)/2}>2,$$ it is enough to show that
$$\big[(\al+\ba)(\al+\ba-1)-\al\ba\big]\left(\frac{\al}{\ba}
\right)^\frac{\ba}{\al+\ba}\frac{1}{\al}>1,$$
since $\al+\ba>2$. Actually, the above expression is equivalent to
$$(\al+\ba)(\al+\ba-1)>\al\ba\bigg[1+ \frac{1}{\al^\frac{\ba}{\al+\ba}\ba^\frac{\al}{\al+\ba}}\bigg].$$
As $\al,\,\ba>1$, a straight forward computation yields
$$
\al\ba\bigg[1+ \frac{1}{\al^\frac{\ba}{\al+\ba}\ba^\frac{\al}{\al+\ba}}\bigg]<2\al\ba\le \frac{(\al+\ba)^2}{2}<(\al+\ba)(\al+\ba-1).
$$
Therefore,   \eqref{2-27-1} implies the existence of
a number $d>0$ such that
	$$
	\inf_{\|(u,v)\|_{H^s\times H^s}  =  r}\tilde J_{f,g}(u,v) >0, \quad  \mbox{provided that } \, 0<\max\{\|f\|_{H^{-s}},\|g\|_{H^{-s}}\}\leq d.
	$$
From here on, proceeding as in the proof of Part $(i)$, with obvious changes, we get the assertion.
\end{proof}
\medskip

{\bf Acknowledgement}:  M. Bhakta wishes to express her sincere gratitude to the Dipartimento di Matematica e Informatica of Universit\`{a} degli Studi di Perugia, where part of this work started during a visit of her to that institution.
The research of M.~Bhakta is partially supported by the {\em SERB MATRICS grant (MTR/2017/000168)}. S.~Chakraborty is supported by {\em NBHM grant 0203/11/2017/RD-II.}

P. Pucci is member of the {\em Gruppo Nazionale per
l'Analisi Ma\-te\-ma\-ti\-ca, la Probabilit\`a e le loro Applicazioni}
(GNAMPA) of the {\em Istituto Nazionale di Alta Matematica} (INdAM).
P. Pucci was also partly supported by of the {\em Fondo Ricerca di Base di Ateneo --
Eser\-ci\-zio 2017--2019} of the University of Perugia, named {\em PDEs and Nonlinear Analysis}.

\end{document}